\documentclass[a4paper,12pt]{article}
\usepackage{amsmath,amssymb,amsthm}
\usepackage[top=2.5cm,bottom=2.5cm,left=2.5cm,right=2.5cm]{geometry}

\usepackage{graphicx}

\newtheorem{proposition}{Proposition}
\newtheorem{theorem}{Theorem}

\begin{document}

\begin{center}
{\bf \Large 
On restrained coalitions in graphs: \\[1mm]  bounds and exact values}
\end{center}

\begin{center}
Andrey A. Dobrynin, Aleksey N. Glebov, H. Golmohammadi

\vspace{2mm}

\emph{Sobolev Institute of Mathematics, 630090, Novosibirsk, Russia}

dobr@math.nsc.ru, angle@math.nsc.ru, h.golmohammadi@g.nsu.ru
\end{center}


\noindent{\bf Abstract:}
A subset $D \subseteq V$ is a dominating set of a graph $G$ with vertex set $V$ if 
every vertex $v \in V \setminus D$ is adjacent to a vertex in $D$.
Two subsets of $V$ form a coalition if neither of them 
is a dominating set, but their union is a dominating set.
A coalition partition of $G$ is its vertex partition $\pi$ such that 
every non-dominating set of $\pi$ is a member of some coalition,
and every dominating set is a single-vertex set in $\pi$.
The coalition number $C(G)$ of a graph $G$ is the maximum cardinality of its 
coalition partitions. 
A subset $R \subseteq V$ is a restrained dominating set 
if $R$ is a dominating set and any vertex of $V \setminus R$  has at least one neighbor in $V \setminus R$.
Restrained dominating coalition, restrained dominating partition and restrained coalition number $RC(G)$ 
are defined by the same way.
In this paper, we prove that $RC(G) \le C(G)$ for an arbitrary graph $G$.
In addition, the restrained coalition numbers of cycles and trees are determined.

\medskip


\

\section{Introduction}
In the present paper, finite simple connected graphs $G(V,E)$ are considered.
The order of a graph is the number of its vertices, $n=|V(G)|$.
The minimum and maximum vertex degree of a graph $G$ is denoted by $\delta(G)$ and $\Delta(G)$, respectively. 
The closed neighborhood of a vertex $v$ is the set $N[v]=\{u \in V(G)\, |\, uv \in E(G)\}$. 
The path, the cycle, the star, and the complete graph of order $n$ are denoted by $P_n$, $C_n$, $S_n$, and $K_n$,
respectively.

A subset $D \subseteq V(G)$ is a dominating set if 
every vertex $v \in V(G) \setminus D$ is adjacent to a vertex of $D$.
The domination theory is a classical branch in the graph theory that  has found 
numerous applications.
For a broader introduction to domination in graphs, we refer the reader to books 
\cite{Hat2020,Hayn2021,Hay2023book,Hay1998,Hen2013}.
A dominating set $D$ is called a restrained dominating set (RD-set) if 
every vertex $v \in V(G) \setminus D$ is adjacent to at least one vertex of  $V(G) \setminus D$,
that is, the induced subgraph on vertices $V(G) \setminus D$ has no isolated vertex.
The restrained domination number $\gamma_r(G)$ is the minimum cardinality
of RD-sets in $G$. The concept of restrained domination was introduced in \cite{Dom1999}. 
The main results on restrained domination up to 2020 were collected in
\cite{Hat2020}. 
So far,  hundreds works have been devoted to this topic (see selected articles 
\cite{Bre2024,Che2026,Dan2006,Nes2025,Zel2005}).

 In \cite{Hay2020}, Haynes et al. introduced the concept of coalitions in graphs.
 Two subsets of $V(G)$ form a dominating coalition if neither of them is dominating, but their union is a dominating set. 
A coalition partition of $G$ is its vertex partition $\pi$ such that 
every non-dominating set of $\pi$  is a member of some coalition, and every dominating
set is a single-vertex set in $\pi$. The coalition number $C(G)$ of a graph $G$ is the maximum
cardinality of its coalition partitions. 
Many variations of the coalition of a graph have been studied in 
\cite{Alik2023,Bakh2023,Dob2025,Dob2024,Hay2023,Hayn2023a,Hayn2021,Hay2023b}.
One of the recent variations is the restrained dominating coalition \cite{Dom1999}.
Let $V_1$  and $V_2$ be two disjoint non-RD-sets of $V(G)$. 
They form a restrained dominating coalition (RD-coalition) if their union is an RD-set. 
A vertex partition $\pi = \{ V_1, V_2,\dots, V_k\}$ of $V(G)$ is called a restrained dominating 
coalition partition (RD-partition) of a graph $G$ if every non-RD-set of $\pi$ is a member of an RD-coalition, 
and every RD-set in $\pi$ is a single-vertex set.
Note that all sets in the RD-partition of $K_n$ are RD-sets.
We only take into account such RD-partitions for non-complete graphs unless $K_n$, since every graph has an 
RD-partition in which all sets are not RD-sets.
The  restrained dominating coalition number $RC(G)$ of  a graph $G$
is the maximum cardinality of RD-partitions of $G$. 
The restrained dominating coalition graph, denoted by $RCG(G,\pi)$, is obtained by associating 
the partition sets of $\pi$ with vertices of the graph, and two vertices  are adjacent if and only 
if the corresponding sets form an RD-coalition in $G$. 
 
In this paper, we consider the restrained dominating coalition number for some classes of graphs.
In particular, we show that $RC(G) \le C(G)$ for an arbitrary graph $G$.
We shall correct the restrained dominating coalition number of cycles  which was obtained in 
\cite{Nes2025}.

\section{Upper bounds on the restrained coalition number}

We show that every RD-partition of a graph $G$ 
can be transformed into a coalition partition of $G$ with larger or equal cardinality.  
This implies that a restrained dominating coalition number cannot exceed the 
coalition number.

\begin{theorem} \label{T1}
Let  $G$ be  a graph of order $n \ge 2$. Then $RC(G) \le C(G)$. 
\end{theorem}

\begin{proof}
Let $\pi = \{ V_1, V_2,\dots, V_k\}$ be an RD-partition of cardinality $k=RC(G)$
and $\pi'$ be a dominating partition of cardinality $k' \ge k$
obtained from $\pi$. 

\smallskip

1. If $G=K_n$, then 
$RC(G)=C(G)=n$. Further, we assume that $G$ is not the complete graph.

\smallskip

2. If each member of $\pi$ is not a  dominating set, then $\pi' = \pi$. 

\smallskip

3. 
Let us show that if the partition $\pi$ contains 
two dominating sets that are not RD-sets, then $k=2$.
Let $V_1$ and $V_2$ be such sets.
Then  there is a vertex $v \not \in V_1$ for which $N(v) \subseteq V_1$ and 
a vertex $u \not \in V_2$ for which $N(u) \subseteq V_2$.
Since $V_1$ is a dominating set, $V_1 \cap N[u] \not = \emptyset$.
It is possible only if $u \in V _1$. By the similar reasoning, $v \in V _2$.
Therefore, $\pi = \{ V_1,V_2 \}$ and $V_1$ forms an RD-coalition with $V_2$.
Indeed, it is easy to see that the union of $X \in \pi \setminus \{V_1,V_2\}$
with any set of $\pi$ is not an RD-set.
Hence, $k=2$. Since $C(G) \ge 2$ for any graph, $RC(G) \le C(G)$.

\smallskip

4. Let the partition $\pi$ contains the unique dominating set  $V_1$ that is not an RD-set.
Then there is a vertex  $v\not \in V_1$ such that $N(v) \subseteq V_1$. 
Suppose that $V_1 \cup V_2$ is an RD-coalition.
It implies that $v\in V_2$, else $N(v) \subseteq   V_1 \cup V_2$
and, therefore, the union  $V_1 \cup V_2$ is not an  RD-coalition. 
Let $X, Y \in \pi \setminus \{V_1,V_2\}$.
Since $N(v) \subseteq V_1\cup X$, the union $V_1 \cup X$ is not an RD-set.
The vertex $v$ is not dominated by $X \cup Y$ for any $Y$.  
Therefore, $X$ must form an RD-coalition only with $V_2$.

\smallskip

4.1. Let the set $V_1$ has no vertex of degree $n-1$. 
Then $V_1$ can be presented as the union  of two disjoint subsets
$V_1=M \cup \{w_1,w_2,\dots,w_m\}$, $m\ge 1$,
where $M$ is a maximal non-dominating set.
Consider  partition
$\pi'=\{ \{w_1\},\{w_2\},\dots,\{w_m\},M,V_2,V_3,\dots,V_k\}$
of cardinality $k+m$.
By construction,  every member of $\pi'$ is a non-dominating set,
and $M \cup \{w_i\}$ is a dominating set for all $i =1,2,\dots,m$ as well as 
$V_2 \cup V_j$ for all $j=3,4,\dots,k$. Hence, $C(G) \ge RC(G)$.

\smallskip

4.2. 
Let the set $V_1$ contains vertices $u_1,u_2,\dots,u_s$ of degree $n-1$,
and the set $W=V_1\setminus \{u_1,u_2,\dots,u_s\}$ has $m$ vertices of degree less than $n-1$, 
where $s \ge 1$ and $m \ge 0$. 

If  $W$ is not a dominating set and  $W \cup V_2$ is a dominating set, then  
 $\pi' = \{ \{u_1\},\dots,\{u_s\},$ $W ,V_2,V_3\dots,V_k\}$ 
 is a desired partition of  cardinality  $k+s > k$.  Hence, $C(G) > RC(G)$.

If  $W \cup V_2$ is not a dominating set, then 
$W\cup V_2\cup V_j$ is a dominating set for all $j \ge 3$.
In this case, we have $\pi'=\{ \{u_1\},\{u_2\},\dots,\{u_s\}, W \cup V_2,V_3,\dots,V_k\}$
of cardinality $s+k-1 \ge k$. Consequently, $C(G) \ge RC(G)$.

If  $W$ is a dominating set, then 
consider partition $\{\{u_1\},\{u_2\},\dots,\{u_s\}, W ,V_2,V_3\dots,V_k\}$
of cardinality $k+s > k$. 
To obtain partition $\pi'$,  we further apply the above splitting procedure to decompose  $W$  
into a maximal non-dominating subset and several single-vertex sets in $\pi'$.
Therefore, $C(G) > RC(G)$.
\end{proof}

From the proof of Theorem~\ref{T1},  an RD-partition $\pi$
of a graph $G$ has at most two dominating sets that are not RD-sets.
If $\pi$ contains two such sets, then $RCG(G)$ is $K_2$.
An example of such a graph is given in 
Fig.~\ref{Fig1}
 (left graph).
The sets $V_1$ and $V_2$ of RD-partition consist of black and white vertices, respectively. 
If $\pi$ has one  dominating set that is not an RD-set, then  the coalition graph is a star  graph.
The right graph in 
Fig.~\ref{Fig1} 
has RD-partition $\pi=\{V_1, V_2,V_3,V_4\}$.
The set  $V_1$ is a dominating set that is not  an RD-set.
The restrained coalition graph is $S_4$.

\begin{figure}[h!]
\centering
\includegraphics[width=0.7\linewidth]{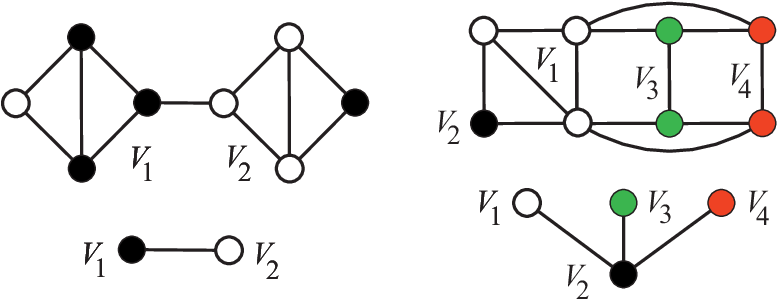}
\caption{Graphs $G$ having dominating sets that are  not RD-sets and their coalition graphs.}  
\label{Fig1}
\end{figure}

By Theorem~\ref{T1}, all upper bounds for $C(G)$ are also  
valid for $RC(G)$. 
It is known that $C(G) \le (\Delta(G)+3)^2/4$ for any graph $G$ \cite{Hayn2021a}.
There are graphs $G$ with $\Delta(G)=3$ (and  $\delta(G)=2$)
for which  $RC(G) = (\Delta+3)^2/4$. For example, a subcubic graph in 
Fig.~\ref{Fig2}
has $RC(G)=9$. Every black vertex of $G$ is a singleton set.

\begin{figure}[h!]
\centering
\includegraphics[width=0.75\linewidth]{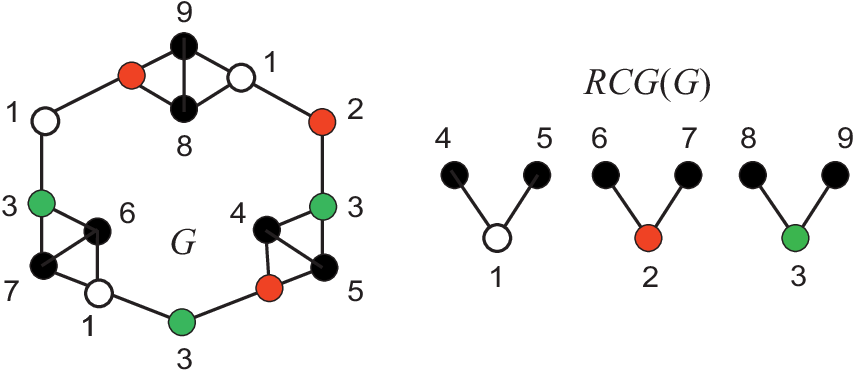}
\caption{Subcubic graph $G$ having $RC(G)=C(G)=9$.}  
\label{Fig2}
\end{figure}

What is the difference between the coalition number and the restrained coalition numbers?
In Table~\ref{Tab1}, the number of connected graphs of order up to 9 vertices with a difference between 
$C(G)$ and $RC(G)$ is displayed.

\begin{table}[h!] 
\centering
\caption{Difference $d = C(G) -RC(G)$ for graphs of order $n$.}  
\begin{tabular}{c|rrrrrrr|r}  \hline
$n$ $\backslash$ $d$ & 0 & 1 & 2  & 3   & 4  & 5  & 6  & graphs  \\    \hline
6  & 77 & 25  & 9  & 1 & .  & .  & .  & .  112 \\ 
7  & 580 & 226  & 43  & 3   & 1  &  .  &  . &  853   \\ 
8  & 8183 & 2399  & 511  & 21   & 2  & 1  & . & 11117   \\ 
9  & 209769 & 41717  & 9169  & 396   & 26  & 2  & 1 & 261080    \\  \hline
\end{tabular}
\label{Tab1}
\end{table}

\begin{proposition} \label{PropCRC}
For every $k\ge 0$, there exists a graph $G$ with $C(G)-RC(G)=k$.
\end{proposition}
\begin{proof}

Consider graph $G$ of order $n$ shown in 
Fig.~\ref{Fig3}.
It is easy to verify that $C(G)=n$.
The sets $V_1, V_2, V_3$ consist of black vertices, 
one red vertex, and white vertices, respectively.
The sets $V_2$ and $V_3$ are not dominating,
and  $V_1$ is a dominating set that is not an RD-set.
The restrained coalition graph of $G$ is $S_3$. 
Now we prove that $RC(G)=3$.
Let $\pi=\{V_1,V_2,V_3,\dots\}$ be an RD-partition of  $G$. 
If vertices $u$ and $v$ belong to the same RD-coalition, 
then this coalition should also contain all vertices 
of degree 2 (see the right graph in 
Fig.~\ref{Fig3}). 
Then $\pi$ consists of at most three sets.
Suppose $u\in V_1$,  $v\in V_2$ and $V_1 \cup V_2$
is not an RD-coalition.
Let $V_3$ be a coalition partner for $V_1$. Since $w$ is a pendant vertex,
$w \in V_1 \cup V_3$.
If $w \in V_1$, then there is no coalition partner for $V_2$ in $\pi$.
Hence, $w\in V_3$.
Let $\pi$ includes a set $V_4$ which has only vertices of degree 2. 
We show that there is no RD-coalition for $V_4$.
Since every coalition must contain the pendant vertex $w$,
the set $V_4$ can  form an RD-coalition only with $V_3$.
If  the union of $V_4$ and $V_3$ does not contain a vertex $x$ of degree 2,
then $V_4 \cup V_3$ does not dominate  $x$. 
If all vertices of degree 2 are in $V_4 \cup V_3$, then 
all neighbors of vertex $u$ belong to the dominating set.
As a result we get  $C(G)-RC(G)=n-3$.
\end{proof}

\begin{figure}[h!]
\centering
\includegraphics[width=0.9\linewidth]{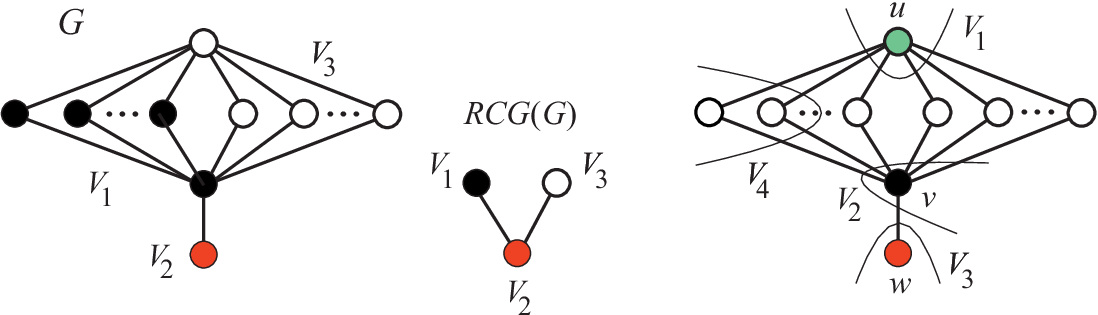}
\caption{Graph $G$ of order $n$ having $C(G)=n$ and $RC(G)=3$.}  
\label{Fig3}
\end{figure}

In the next result, we establish an upper bound on the restrained coalition number 
in terms of the restrained domination number ${\gamma}_r(G)$.

\begin{proposition} \label{gamma}
Let $G$ be a graph of order $n\ge 2$ and  $\gamma_r(G) \ge 2$. Then 
$$
RC(G) \le n -  \gamma_r(G)+2
$$
and the bound is sharp for complete bipartite graph $K_{r,s}$, $r\ge s\ge 2$.
\end{proposition}
\begin{proof}
Let  $\pi(G)=\{ V_1,V_2,\dots,V_k \}$ be an RD-partition and $k=RC(G)$. 
Since $\gamma_r(G)\ge 2$, the partition does not contain a singleton RD-set.  
Assume that $V_1$ and $V_2$ form an RD-coalition.
Then $V_1\cup V_2$ is an RD-set of $G$, and thus
 $\gamma_r(G)\leq \left \vert V_1 \right \vert +\left \vert V_2 \right \vert $.
Accordingly,
$$
n   = \left\vert V_1 \right\vert + \left\vert V_2 \right \vert  + \left \vert V_3 \right \vert + \cdots + \left \vert V_k \right \vert  
\geq \left \vert V_1 \right \vert  + \left \vert V_2 \right \vert + (k-2) \geq \gamma_r(G)+(k-2)
$$
leading to the desired upper bound.
\end{proof}

The similar bound is also valid for the dominating coalition number 
based on different types of domination in graphs.

\section{Restrained coalition partitions of trees}

It is known the following upper bound for graphs with pendant vertices.

\begin{proposition} {\rm \cite{Nes2025}} \label{CNdeg1}
Let $G$ be a graph with $\delta(G) =1$. Then $RC(G) \le \Delta(G) + 2$.
\end{proposition}

The set of restrained coalition graphs of paths contains only two graphs.

\begin{proposition} {\rm \cite{Nes2025}} \label{CNpath}
Path $P_n$ of order $n$ has coalition graphs  $S_2$ 
for $2 \le n \le 5$ and $S_3$ for $n\ge 6$.
\end{proposition}

It is easy to see that all pendant vertices of a tree $T$ must belong to every RD-set.
This immediately implies that if $RC(T) \ge 3$, then all pendant vertices belong to one set of an RD-partition,
and every RD-coalition includes this set.   Therefore, $RCG(T)$ of a tree $T$ is always a star graph.

\begin{proposition} \label{Tree1}
For every  $n\ge 13$ and every $\Delta \ge 3$, there exist trees $T$  of order $n$ with the maximal degree 
$\Delta(T)=\Delta$  such that the restrained coalition graph of $T$ is $S_{\Delta + 2}$. 
\end{proposition}
\begin{proof}

Consider a tree $T$ with $\Delta(T) \ge 3$ and its RD-partition shown in 
Fig.~\ref{Fig4}. 
By increasing the order of the bottom  path with white vertices, we have 
a tree of an arbitrary order $n\ge 13$ with given $\Delta$.
The corresponding coalition graph is the star graph $S_{\Delta + 2}$. 
\end{proof}

\begin{figure}[h!]
\centering
\includegraphics[width=0.75\linewidth]{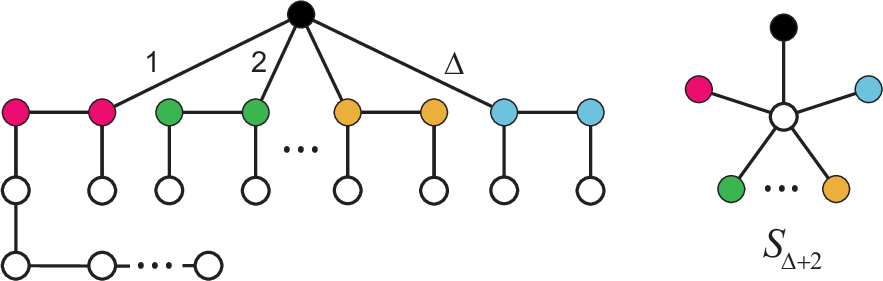}
\caption{A tree and its coalition graph.}  
\label{Fig4}
\end{figure}

\newpage

The numbers of trees with given restrained coalition number are presented in 
Table~\ref{Tab2}.

\begin{table}[h!] 
\centering
\caption{Number of trees of order $n$ with given $RC(T)$.} 
\begin{tabular}{c|rrrrrrrrrr}  \hline
$RC(T)$ $\setminus$ $n$ & 4  & 5   & 6   & 7    & 8     &  9   & 10  &  11  & 12  & 13 \\  \hline
2                     & 2  & 3   & 5   &  8   & 13    & 20   & 34  &  54  & 95  &  160    \\
3                     &  . &  .  &  1  &  3   &  10   &  26  &  67 &  155 & 358 &  792 \\
4                     &  . &  .  &  .  &  .   &  .    &  1   &  5  &  26  & 98  &  348 \\
5                     &  . &  .  &  .  &  .   &  .    &  .   &  .  &   .  & .   &   1 \\  \hline
total                     &  2 &  3  &  6  &  11  &  23   &  47  & 106 &  235 & 551   & 1301 \\
\hline
\end{tabular}
\label{Tab2}
\end{table}

\begin{proposition} \label{PropCRC-2}
For every $k\ge 0$, there exists a tree $T$ with $C(T)-RC(T)=k$.
\end{proposition}
\begin{proof}
For $k=0,1,2$, we have $C(P_2)-RC(P_2)=2-2=0$, 
$C(P_3)-RC(P_3)=3-2=1$, and $C(P_6)-RC(P_6)=5-3=2$
\cite{Hay2020,Nes2025}.

Consider tree $T$ of order $n \ge 7$ having 
$\Delta(T)= \deg(v) \ge 3$ depicted in 
Fig.~\ref{Fig5}.
We show that $C(T)=\Delta(T)+2$.   
Suppose that $C(T) > \Delta(T)+2 \ge 5$. Then vertices of degree 1 and 2 are contained in at 
least $\Delta(T)+2$ different sets of the coalition partition.
 Thus, there exist two pendant vertices $x$ and $y$ such that the vertices of 
 $N[x] \cup N[y]$ are contained in four different sets of the coalition partition. 
 Since every coalition must contain a vertex of  $N[x]$ and a vertex of $N[y]$, we have $C(T) \le 4$, a contradiction.
 The following vertex partition of  $T$ has cardinality $\Delta(T)+2$:   vertex $v$ and every vertex of degree 2
are singleton sets, and all pendant vertices form one set.

Assume that $V_1, V_2$ are sets of an RD-partition $\pi$ of $T$ such that 
$v\in V_1$ and $V_1 \cup V_2$ form an RD-coalition.
Since the set of all pendant vertices $X_e$ of a graph must belong to every RD-coalition,
$X_e \subset V_1 \cup V_2$. 
If there exists $V_3 \in \pi$ and $V_3$ contains a vertex of degree 2,
then $V_1 \cup V_2$  is not an RD-coalition (all neighbors of this vertex will be in the dominating set).
Hence, the set $V_1 \cup V_2$ includes all vertices of $T$, and we get $RC(T)=2$.
Then $C(T)-RC(T)=\Delta(T)\ge 3 $.
\end{proof}

\begin{figure}[h!]
\centering
\includegraphics[width=0.2\linewidth]{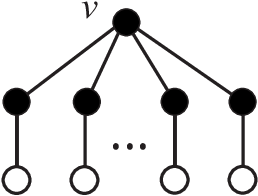}
\caption{Tree $T$ having $C(T)-RC(T)=\Delta(T)$.}  
\label{Fig5}
\end{figure}

Table~\ref{Tab3} shows  the numbers of trees $T$ of order up to 13 with the given difference 
of the  coalition numbers.

\begin{table}[h!] 
\centering
\caption{Difference $d = C(T) -RC(T)$ for trees of order $n$.} 
\begin{tabular}{c|rrrrrrrrrrrr}  \hline
$d$ $\backslash$ $n$ & 3 & 4  & 5 & 6  & 7  & 8  & 9  & 10 & 11  & 12  & 13  \\    \hline
0                      & 0 & 0  & 0 & 0  & 0  & 0  & 1  & 3  & 13  & 36  &  105 &   \\ 
1                      & 1 & 1  & 1 & 1  & 2  & 5  & 11 & 26 & 62  & 160 &  421 &   \\ 
2                      & . & 1  & 2 & 5  & 8  & 17 & 32 & 70 & 143 & 305 &  635 &   \\ 
3                      & . & .  & . & .  & 1  & 1  & 2  & 6  & 14  & 45  &  126 &   \\ 
4                      & . & .  & . & .  & .  & .  & 1  & 1  & 2   & 4   &   11 &   \\ 
5                      & . & .  & . & .  & .  & .  & .  & .  & 1   & 1   &   2  &   \\ 
6                      & . & .  & . & .  & .  & .  & .  & .  & .   & .   &  1  &   \\  \hline
trees                  & 1 & 2  & 3 & 6  & 11 & 23 & 47 & 106& 235 & 551 & 1301 &   \\  \hline
\end{tabular}
\label{Tab3}
\end{table}

\section{Restrained coalition number of cycles}

It is known that for  the dominating coalition number of cycles $C(C_n) \le 6$ \cite{Hay2020}.  
By Theorem~\ref{T1}, we have $RC(C_n) \le 6$. 
Restrained coalition numbers of cycles were previously reported in \cite{Nes2025}.
Here we formulate a correct version of their result.

\begin{proposition} \label{HE}
For the restrained domination coalition number of cycle $C_n$,
$$
RC(C_n)= \left\{
\begin{array}{ll}
3, & {\rm if \ } n=3,5 \\
4, & {\rm if \ } n=4,8 \\
5, & {\rm if \ } n=7,11 \\
6, & {\rm otherwise.} \\
\end{array}
\right.
$$
\end{proposition}
\begin{proof}
It is easy to verify that $RC(C_3)=RC(C_5)=3$, $RC(C_4)=4$, and $RC(C_6)=6$.
The restrained coalition numbers of cycles of order $n=7,8,11$    
can be found by computer (generation and check of all RD-partitions).
Examples of suitable RD-partitions and 
the corresponding  restrained coalition graphs are presented in 
Figs.~\ref{Fig6},\,\ref{Fig7}.
For cycles with $RC(C_n) = 6$, RD-partitions  
are shown in 
Fig.~\ref{Fig8} ($n=9,12,15$) and Fig.~\ref{Fig9} 
(other cycles).
The order of cycles in 
Fig.~\ref{Fig9} 
is equal to $n=3m+4k+7$, where $m \ge 1$ and $k \ge 0$. 
It is not hard to check that $n$ goes over
all values $10, 13, 14, 16, 17, \dots$ \ .
\end{proof}

\begin{figure}[h!]
\centering
\includegraphics[width=0.45\linewidth]{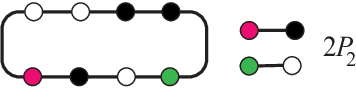}
\label{treeCRC}
\caption{Cycle $C_8$ having restrained coalition number 4.}  
\label{Fig6}
\end{figure}

\vspace{-3mm}

\begin{figure}[h!]
\centering
\includegraphics[width=0.6\linewidth]{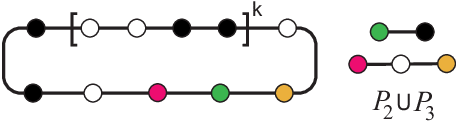}
\caption{Cycles $C_7$ and $C_{11}$ having restrained coalition number 5, $k = 0,1$.}  
\label{Fig7}
\end{figure}

\vspace{-3mm}

\begin{figure}[h!]
\centering
\includegraphics[width=0.55\linewidth]{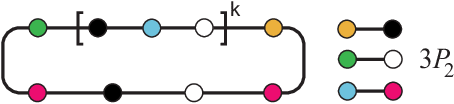}
\caption{Cycles $C_9, C_{12}$ and $C_{15}$ having restrained coalition number 6, $k =1,2,3$.}  
\label{Fig8}
\end{figure}

\newpage

\begin{figure}[h!]
\centering
\includegraphics[width=0.75\linewidth]{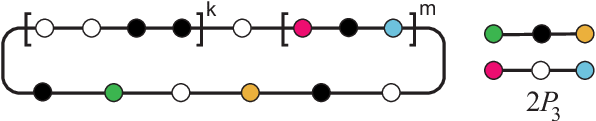}
\caption{Cycles $C_n$ of order $n=3m+4k+7$ having $RC(C_n) =6$, $m \ge 1$, $k \ge 0$.}  
\label{Fig9}
\end{figure}

\

\textbf{ Acknowledgment.}
This study was supported by the state contract of the Sobolev Institute of Mathematics (project number FWNF-2026-0011).

\end{document}